\documentclass[a4paper,leqno]{amsart}

\usepackage{rotating}

\usepackage{amssymb,amsthm,amsmath,amsrefs}

\usepackage{ifthen}

\newcounter{mylisti} \newcounter{mylistii}
\newcounter{nest}
\newcommand{\defaultlabel}{}

\newcommand{\bn}{\ensuremath{\mathbb N}}

\newcommand{\br}{\ensuremath{\mathbb R}}

\newcommand{\cB}{\ensuremath{\mathcal B}}
\newcommand{\cC}{\ensuremath{\mathcal C}}

\newcommand{\cF}{\ensuremath{\mathcal F}}
\newcommand{\cG}{\ensuremath{\mathcal G}}

\newcommand{\cM}{\ensuremath{\mathcal M}}

\newcommand{\cO}{\ensuremath{\mathcal O}}
\newcommand{\cP}{\ensuremath{\mathcal P}}
\newcommand{\cR}{\ensuremath{\mathcal R}}
\newcommand{\cS}{\ensuremath{\mathcal S}}
\newcommand{\cT}{\ensuremath{\mathcal T}}
\newcommand{\cU}{\ensuremath{\mathcal U}}

\newcommand{\cX}{\ensuremath{\mathcal X}}
\newcommand{\cY}{\ensuremath{\mathcal Y}}

\newcommand{\Deltat}{\ensuremath{\widetilde{\Delta}}}

\newcommand{\abs}[1]{\lvert #1\rvert}
\newcommand{\bigabs}[1]{\big\lvert #1\big\rvert}
\newcommand{\Bigabs}[1]{\Big\lvert #1\Big\rvert}

\newcommand{\cbi}{\ensuremath{\mathrm{I}_{\mathrm{CB}}}}

\newcommand{\dist}{\ensuremath{\mathrm{dist}}}

\newcommand{\fin}[1]{{[#1]}^{<\omega}}

\newcommand{\sizelek}[1]{{[#1]}^{\leq k}}

\newcommand{\norm}[1]{\lVert #1\rVert}
\newcommand{\bignorm}[1]{\big\lVert #1\big\rVert}
\newcommand{\Bignorm}[1]{\Big\lVert #1\Big\rVert}

\newcommand{\symdif}{\bigtriangleup}

\newcommand{\supp}{\operatorname{supp}}

\newcommand{\co}{\mathrm{c}_0}
\newcommand{\coo}{\mathrm{c}_{00}}
\newcommand{\mc}{\mathrm{c}}

\newcommand{\E}{\exists\,}

\newcommand{\vare}{\varepsilon}
\newcommand{\varf}{\varphi}

\newcommand{\ds}{\displaystyle}

\newcommand{\phtm}[1]{\text{\makebox[0pt]{\phantom{$#1$}}}}

\newcommand{\ie}{\textit{i.e.,}\ }

\newtheorem{thm}{Theorem}

\newtheorem{prop}[thm]{Proposition}
\newtheorem{cor}[thm]{Corollary}

\newtheorem{qn}{Question}

\theoremstyle{definition}

\theoremstyle{remark}

\newtheorem*{rem}{Remark}

\newcommand{\sepm}{\ensuremath{\cM}}

\newcommand{\sepb}{\ensuremath{\cS\cB}}
\newcommand{\cotype}{\ensuremath{\cC\cO\cT}}
\newcommand{\type}{\ensuremath{\cT\cY\cP}}
\newcommand{\sepref}{\ensuremath{\cS\cR}}

\newcommand{\lip}{\ensuremath{\mathrm{Lip}}}

\newcommand{\lipembed}{\ensuremath{\underset{\lip}{\lhook\joinrel\relbar\joinrel\rightarrow}}}
\newcommand{\clipembed}{\ensuremath{\underset{C-\lip}{\lhook\joinrel\relbar\joinrel\rightarrow}}}
\newcommand{\aiembed}{\ensuremath{\underset{\ensuremath{\mathrm{a.i.}}}{\lhook\joinrel\relbar\joinrel\rightarrow}}}
\newcommand{\isometric}{\ensuremath{\underset{=}{\lhook\joinrel\relbar\joinrel\rightarrow}}}
\newcommand{\uponelip}{\begin{sideways}\ensuremath{\underset{1-\textrm{Lip}}{\lhook\joinrel\relbar\joinrel\relbar\joinrel\relbar\joinrel\relbar\joinrel\rightarrow}}\end{sideways}}
\newcommand{\uptwolip}{\begin{sideways}\ensuremath{\underset{2-\textrm{Lip}}{\lhook\joinrel\relbar\joinrel\relbar\joinrel\relbar\joinrel\relbar\joinrel\rightarrow}}\end{sideways}}
\newcommand{\updonelip}{\begin{sideways}\ensuremath{\underset{D_1-\textrm{Lip}}{\lhook\joinrel\relbar\joinrel\relbar\joinrel\relbar\joinrel\relbar\joinrel\rightarrow}}\end{sideways}}
\newcommand{\updklip}{\begin{sideways}\ensuremath{\underset{D_k-\textrm{Lip}}{\lhook\joinrel\relbar\joinrel\relbar\joinrel\relbar\joinrel\relbar\joinrel\rightarrow}}\end{sideways}}
\newcommand{\updwlip}{\begin{sideways}\ensuremath{\underset{D_{\omega}-\textrm{Lip}}{\lhook\joinrel\relbar\joinrel\relbar\joinrel\relbar\joinrel\relbar\joinrel\rightarrow}}\end{sideways}}

\renewcommand{\cbi}{\ensuremath{i_{\mathrm{CB}}}}

\begin{document}

\title[Metric geometry of the Hamming cube]{The metric geometry of the
  Hamming cube and applications}

\author{F.~Baudier}
\address{Department of Mathematics, Texas A\&M University, College
  Station, TX 77843, USA and Institut de Math\'ematiques Jussieu-Paris
  Rive Gauche, Universit\'e Pierre et Marie Curie, Paris, France}
\email{florent@math.tamu.edu}

\author{D.~Freeman}
\address{Department of Mathematics and Computer Science, 220 N. Grand
  Blvd., St.Louis, MO 63103, USA}
\email{dfreema7@slu.edu}

\author{Th.~Schlumprecht}
\address{Department of Mathematics, Texas A\&M University, College
  Station, TX 77843, USA and Faculty of Electrical Engineering, Czech
  Technical University in Prague,  Zikova 4, 166 27, Prague, Czech
  Republic}
\email{schlump@math.tamu.edu}

\author{A.~Zs\'ak}
\address{Peterhouse, Cambridge, CB2 1RD, UK}
\email{a.zsak@dpmms.cam.ac.uk}

\date{18 March 2014}

\thanks{The first author's research was supported by ANR PDOC, project
  NoLiGeA. The second author's research was supported by NSF grant
  DMS1332255. The third author's research was supported by NSF grant
  DMS1160633. The second and fourth authors were supported by the
  Workshop in Analysis and Probability at Texas A\&M University in
  2013. The fourth author was supported by Texas A\&M University while
  he was Visiting Scholar there in 2014.}

\begin{abstract} The Lipschitz geometry of segments of the infinite
  Hamming cube is studied. Tight estimates on the distortion necessary
  to embed the segments into spaces of continuous functions on
  countable compact metric spaces are given. As an application, the
  first nontrivial lower bounds on the $C(K)$-distortion of important
  classes of separable Banach spaces, where $K$ is a countable compact
  space in the family
  \[
  \{ [0,\omega],[0,\omega\cdot 2],\dots,
  [0,\omega^2], \dots, [0,\omega^k\cdot n],\dots,[0,\omega^\omega]\}\ ,
  \]
  are obtained.
\end{abstract}

\maketitle

\section{introduction}

\subsection{Motivation and Background}\label{intro}

Assume that one is given a Banach space $Y$ and a class $\cC$ of metric
spaces. Given an \emph{arbitrary }metric space $M$ in the
class $\cC$, it is natural to study the smallest distortion
achievable when trying to embed $M$ into $Y$ through a bi-Lipschitz
embedding. This quite general, quantitative embedding problem is an
important topic in the nonlinear geometry of Banach Spaces. When $Y$
is a Hilbert space this problem is known as estimating the Euclidean
distortion of the given class. It is well recognized that being able
to accurately estimate the Euclidean distortion of some specific
classes of metric spaces has tremendous and far reaching
applications in both mathematics and computer science. In this paper
we consider the general embedding problem when $Y$ is the Banach
space $C(K)$, the space of continuous functions on a compact
topological space $K$. We will mainly stay in the separable world
and therefore consider only compact \emph{metric }spaces $K$ as well
as classes $\cC$ contained in the class $\sepm$ of separable metric spaces.
The theory is clearly isometric and,
although $\co$ is not isometric to a $C(K)$-space for any compact space
$K$, embeddings into $\co$ are related to those into $C(K)$-spaces.
Indeed $\co$ is a hyperplane of the space $\mc$ of convergent sequences of
real numbers, which can be seen as the space
$C(K)$ where $K=\gamma\bn$ is the Alexandrov-compactification (or
one-point compactification) of $\bn$. Moreover, it is
easy to show that whenever $K$ is an infinite (not necessarily metrizable)
compact Hausdorff space, $C(K)$ contains a subspace isometric
to $\co$ (see~\cite{AlbiacKalton2006}*{Proposition 4.3.11}). We browse
briefly and chronologically through a few classical and historical embedding
results into $C(K)$-spaces and $\co$. Back in 1906, Fr\'echet
observed~\cite{Frechet1906} that every separable metric space admits
an isometric embedding into the space $\ell_\infty(\bn)$. An easy
application of the Hahn-Banach theorem gives a linear isometric
embedding of every separable Banach space into $\ell_\infty(\bn)$.
These results can actually be cast as embedding results into a
$C(K)$-space. Indeed $\ell_\infty(\bn)$ can be identified with the
space $C(\beta\bn)$ where $\beta\bn$ denotes the Stone-\v{C}ech
compactification of $\bn$. Note that $\beta\bn$ is an uncountable
compact space and since $\ell_\infty(\bn)$ is nonseparable, $\beta\bn$
is not metrizable. The Banach-Mazur theorem~\cite{BanachMazur33}
asserts that every separable Banach space admits an (linear)
isometric embedding into the space $C[0,1]$. Note that $[0,1]$
equipped with its canonical distance is compact and hence $C[0,1]$ is
separable. With the help of Fr\'echet's embedding, it is easily seen
that every separable metric space can be isometrically embedded into
$C[0,1]$. In 1974 Aharoni proved in~\cite{Aharoni1974} that the
$\co^+$-distortion of every separable metric space is less than $6$.
In that same paper he also proved that the $\co$-distortion of
$\ell_1$ is at least $2$. A few years later Assouad~\cite{Assouad1978}
showed that the $\co^+$-distortion of every
separable metric space is at most $3$. The fact that there is a bi-Lipschitz
embedding  with distortion exactly $3$ and that this value is
optimal for embeddings into $\co^+$ is due to
Pelant~\cite{Pelant1994}. Finally the end of the story regarding
embeddings
into $\co$ was completed by Kalton and
Lancien~\cite{KaltonLancien2008} when they constructed an embedding
with distortion $2$ (respectively, $1$) for every separable, respectively
proper metric space. Recall that a metric space is \emph{proper }if
all its closed balls are compact.

\subsection{Notation and Definitions}

Let $M$ and $N$ be two metric spaces. Define the \emph{distortion }of
a map $f\colon M\to N$ to be
\[
\dist(f):= \norm{f}_{\lip}\norm{f^{-1}}_{\lip}=\Big(\sup_{x\neq y \in
M}\frac{d_N(f(x),f(y))}{d_M(x,y)}\Big)\Big(\sup_{x\neq y \in
M}\frac{d_M(x,y)}{d_N(f(x),f(y))}\Big)\ .
\]
If the distortion of
$f$ is finite, $f$ is said to be a \emph{bi-Lipschitz embedding}. The
convenient notation $M\lipembed N$ means that
there exists a bi-Lipschitz embedding $f$ from $M$ into $N$. We are concerned
with the quantitative theory, and if $\dist(f)\leq C$, we use the
notation $M \clipembed N$. The parameter
$c_{N}(M)=\inf\{C\geq 1:\,M\clipembed N\}$ will
be referred to as \emph{the $N$-distortion of $M$}.

Let $\cF$ be a collection of metric spaces. We can define \emph{the
  $N$-distortion of the class $\cF$ }as follows:
\[
c_{N}(\cF)=\sup\{c_{N}(M):\,M\in \cF\}\ .
\]
Finally, for two families $\cF$ and $\cG$ of metric spaces we define
\[
c_{\cG}(\cF)=\sup_{M\in\cF}\, \underset{N\in\cG}{\inf\phtm{p}}\,
c_{N}(M)\ .
\]
As an application of our work on the metric geometry of the Hamming
cube we will give nontrivial estimates on the parameter $c_{N}(\cF)$
for the following spaces and classes:
\begin{itemize}
\item $N=C(K)$ for some countable compact metric space $K$.
\item $\cF$ is one of the following classes:
\begin{enumerate}
\item $\sepm:=\{M:\,M\text{ separable metric space}\}$
\item $\sepb:=\{X:\, X\text{ separable Banach space}\}$
\item $\cotype:=\{X:\, X\text{ separable Banach space with nontrivial cotype}\}$
\item $\type:=\{X:\, X\text{ separable Banach space with nontrivial type}\}$
\item $\sepref:=\{X:\, X\text{ separable, reflexive Banach space}\}$.
\end{enumerate}
\end{itemize}
Observe that $c_{N}(\sepb)=c_{N}(\sepm)$. Indeed, it is clear that
$c_{N}(\sepb)\leq c_{N}(\sepm)$, and the reverse inequality follows
from the fact that every separable metric space embeds isometrically
into the separable Banach space $C[0,1]$.

\subsection{Stratification of the Hamming cube}

We define a \emph{stratification }of a metric space $M$ to be a
sequence $M_1\subset M_2\subset \dots$ of subsets of $M$ such that
$M=\bigcup _{k=1}^\infty M_k$. (More generally, it is a way of
expressing $M$ as a direct limit of metric spaces, but this
generality will not be needed here.) The sets $M_k$ are the
\emph{segments }of $M$, and the sets $M_k\setminus M_{k-1}$ are the
\emph{layers }of $M$ (where we put $M_0=\emptyset$). In this paper
we are concerned with stratifications of
the Hamming cube. The infinite Hamming cube $H_\infty$
is the set of all infinite sequences in $\{0,1\}$  containing only
finitely many $1$s equipped with the Hamming  metric $d_H$, where
$d_H(\sigma,\tau)=\abs{\{i\in\bn:\, \sigma_i\neq\tau_i\}}$. It is
isometric to the metric space $\Delta_\infty$ consisting of the set
$\fin{\bn}$ of all finite subsets of $\bn$ equipped with the symmetric
difference metric $d_\Delta$, where $d_\Delta(A,B)=\abs{A\symdif
  B}$. The isometry between $\Delta_\infty$ and $H_\infty$ is the
natural one identifying a set with its indicator function.

We describe two natural stratifications of the infinite Hamming
cube. For $k\in\bn$ let $H_k=\{0,1\}^k$ thought of as a subset of
$H_\infty$ by extending elements of $H_k$ to infinite binary sequences
with the addition of an infinite tail of $0$s. The layers of the
stratification $(H_k)_{k=0}^\infty$ are $\{\emptyset\}$ and families
of sets of the form $\{ A\subset\bn:\,\max A=n\}$, $n\in\bn$.
The members of the second stratification are the
families $\Delta_k=\sizelek{\bn}$ of subsets of $\bn$ of size at
most~$k$.  The set  $\Delta_k$ can be identified with the rooted
countably infinitely branching tree of height $k$. Note, however, that
the metric $d_\symdif$ is not the classical graph metric of a tree.

The two stratifications share some essential metric
properties despite being quite different from the combinatorial and
structural standpoint. For example, $\Delta_k$ (respectively, $H_k$)
is a $2k$-bounded (respectively, $k$-bounded), $1$-separated metric
space. However, $\Delta_k$ is a countable non-proper metric space while
$H_k$ is a finite metric space. The two stratifications are different
in the Lipschitz category in the following sense. Two families of
metric spaces $\cF$ and $\cG$ shall be called \emph{Lipschitz
  equivalent }if $c_{\cF}(\cG)c_{\cG}(\cF)<\infty$. The
stratifications $\cO=(H_k)_{k\geq 0}$ and $\cU=(\Delta_k)_{k\geq 0}$ are
not Lipschitz equivalent. Indeed, the embedding
$(\sigma_1,\cdots,\sigma_k)\mapsto \big\{
i\in\{1,\dots,k\}:\,\sigma_i=1\big\}$ sends $H_k$ isometrically into
$\Delta_k$ (\ie $c_{\cU}(\cO)=1$), however, $c_{\cO}(\cU)=\infty$
since it is impossible to embed a single $\Delta_k$ bi-Lipschitzly
into any $H_i$ because of a cardinality obstruction (assuming $k\geq
1$ of course).

Sometimes metric information about a stratification can be used to
derive metric information on the stratified space and
vice-versa. However, this need not be the case. As we will see
$H_\infty$ does not embed isometrically into $\co$ and this will be
witnessed by $\Delta_2$. This is in stark contrast with the fact that
every $H_k$, as any finite metric space, embeds isometrically into
$\co$. So in some sense $(\Delta_k)_{k\geq 0}$ captures more of the
structure of $\Delta_\infty$.

\subsection{Organization of the paper.}\label{organization}
From now on we will consider \emph{countable }compact
metric spaces and we will focus on the following nested family:
\[
[0,\omega]\subset[0,\omega\cdot 2] \subset
\dots\subset[0,\omega^2]\subset\dots\subset[0,\omega^\alpha\cdot
  n]\subset\cdots\subset[0,\omega^\omega]\ ,
\]
where, as usual, $\omega$ is the first infinite ordinal. It is a
simple fact that if compact spaces $K$ and $L$ are homeomorphic
then the Banach spaces $C(K)$ and $ C(L)$ are isometrically
isomorphic. Note that the converse is also true by the Banach-Stone
theorem.  Therefore the $C(K)$-spaces arising from the nested family
above are mutually non isometric Banach spaces.
However, this family has the property that $C(K)$ embeds linearly
isometrically into $ C(L)$ whenever $K\subset L$ since then $K$ is
in fact a clopen subset of $L$.

In Section~\ref{upperbound} we estimate from above the
$C(K)$-distortion of the infinite Hamming cube and its stratification
$\Delta_k$. We will show that when $1\leq r\leq k<\infty$, then
$c_{C([0,\omega^r])}(\Delta_k)\leq\min\big\{\frac{k}{r},2\big\}$. In
particular, $\Delta_k$ embeds isometrically into $C([0,\omega^k])$.
In Section~\ref{lowerbound} we will give lower bounds. To estimate
$c_{C([0,\omega^r])}(\Delta_k)$ from below, we exhibit a connection
between a topological property of the compact space $K$ and the
$C(K)$-distortion of the metric spaces $\Delta_k$. Roughly speaking,
if the compact metric space $K$ is small in the sense of the
Cantor-Bendixson derivation, then the $C(K)$-distortion of $\Delta_k$
cannot be too small. More precisely, we show that if the
Cantor-Bendixson index of $K$ is $k\geq 2$, then
$c_{C(K)}(\Delta_k)\geq \frac{k}{k-1}$. In Section~\ref{applications}
we give some applications concerning the parameters $c_{C(K)}(\sepm)$,
$c_{C(K)}(\sepb)$, $c_{C(K)}(\cotype)$ and $c_{C(K)}(\sepref)$. We
conclude with a few open questions that arise naturally from our work.

\section{Low distortion embeddings of the Hamming cube}\label{upperbound}

\subsection{Embeddings of the sets $\boldsymbol{\Delta_k}$.}
We will show, by constructing suitable bi-Lipschitz embeddings, that
when $1\leq r\leq k<\infty$, then
$c_{C([0,\omega^r])}(\Delta_k)\leq\min\big\{\frac{k}{r},2\big\}$.  In
particular, $\Delta_k$ embeds isometrically into $C([0,\omega^k])$,
and hence also into $C([0,\omega^r])$ for $r\geq k$.

We will need a description of $C(K)$-spaces as \emph{tree
  spaces}, due to Bourgain~\cite{bourgain:79} (see
also~\cite{odell:04}), which we now proceed to describe. Recall that a
\emph{tree }is a set $T$ with a partial
order $\preccurlyeq$ such that $b_t=\{ s\in T:\,s\preccurlyeq t\}$ is
finite and linearly ordered by~$\preccurlyeq$ for all $t\in T$. The
space $\coo(T)$ consists of all functions $x\colon T\to\br$ with $\{
t\in T:\,x(t)\neq 0\}$ is finite. The unit vector basis $(e_t)_{t\in
  T}$ of $\coo(T)$ consists of functions $e_t$ taking the value~$1$
at~$t$ and $0$ everywhere else. For $t\in T$ the functional
$\beta_t$ is defined by summing along the branch $b_t$:
\[
\beta_t(x)=\sum _{s\in b_t} x(s)\qquad (x\in \coo(T))\ .
\]
We define a norm $\norm{\cdot}$ on $\coo(T)$ by letting
\[
\norm{x}=\sup_{t\in T} \abs{\beta_t(x)}\qquad (x\in \coo(T))\ .
\]
The tree space corresponding to $T$ is the completion $S(T)$ of
$(\coo(T),\norm{\cdot})$. It is easy to verify that $(e_t)$ is a
normalized, monotone basis of $S(T)$. Note that the branch functionals
can be expressed in terms of the biorthogonal functional as follows:
$\beta_t=\sum_{s\in b_t} e^*_s$. We now let $K$ be the $w^*$-closure
in $S(T)^*$ of the set $\{\beta_t:\, t\in T\}$. This is a compact
Hausdorff space and $0\in K$ if and only if $T$ has infinitely many
\emph{initial nodes }(\ie elements $t\in T$ for which $s\preccurlyeq
t$ implies $s=t$). The restriction to $K$ of the canonical embedding of
$S(T)$ into $S(T)^{**}$ is an isometric isomorphism $S(T)\to C(K)$. By
the Stone-Weierstrass theorem it is onto $C(K)$ if $0\notin K$ and
onto $C_0(K)$ (functions vanishing at $0$) if $0\in K$. It turns out
that every $C(K)$-space with separable dual can be represented as a
tree space but we will not need this result in its full generality. We
will now mention the examples relevant to us.

For each $k\in\bn$ let $T_k$ be the tree
$\big(\sizelek{\bn},\preccurlyeq\big)$, where
$s\preccurlyeq t$ if and only if $s$ is an initial segment of
$t$. Thus $T_k$ is the rooted, countably infinitely branching tree of
height~$k$. As usual, we identify a set $t\subset \bn$ with the
sequence $i_1,i_2,\dots$, where $i_1<i_2<\dots$ are the elements of
$t$. So, for example, we shall write $e_m$ for the basis element
$e_{\{m\}}$ of $S(T_k)$, etc. The set $\{\beta_t:\,t\in T_k\}$ is
homeomorphic to $(0,\omega^k]$ (and hence to $[0,\omega^k]$) via the
map $\beta_\emptyset\mapsto \omega^k$ and
\[
(i_1,\dots,i_r)\mapsto \sum_{j=1}^{r-1} \omega
^{k-j}(i_j-i_{j-1}-1)+\omega^{k-r}(i_r-i_{r-1})\ ,
\]
for $1\leq r\leq k,\ i_1<\dots<i_r$ (and with $i_0=0$). Thus
$S(T_k)\cong C([0,\omega^k])$.
Let us now denote by $T$ the disjoint union of the trees $T_k$. For
$s,t\in T$ we have $s\preccurlyeq t$ if and only if for some $k$ both
$s$ and $t$ belong to $T_k$ and $s\preccurlyeq t$ in $T_k$. The tree
space $S(T)$ is then isometrically
isomorphic to $C_0([0,\omega^\omega))$ which of course isometrically
embeds into $C([0,\omega^\omega])$. Note also that $S(T)\cong
\big(\bigoplus_{k=1}^\infty S(T_k)\big)_{\co}$. For the rest of the
paper we fix $T_k$ and $T$ to be trees just described.
\begin{thm}
  \label{thm:Delta_k-into-ST_r}
  For every $1\leq r\leq k$ there exist a map $\varf_{k,r}\colon
  \Delta_k\to  C([0,\omega^r])$ such that
  $\dist(\varf_{k,r})\leq\frac{k}{r}$. It follows that
  $c_{C([0,\omega^r])}(\Delta_k)\leq \min \big\{ \frac{k}{r},2\big\}$.
\end{thm}
\begin{proof}
  For each $r\in\bn$ we define the map
  \[
  f_r\colon \bn \to S(T_r)\ ,\qquad m \mapsto -\sum_{i=1}^{m-1} e_i +
  e_m + 2\sum_{j=2}^r\ \sum_{%
    \begin{subarray}{c}
      i_1<\dots <i_j\\
      i_j=m
    \end{subarray}%
  } e_{i_1,\dots,i_j}\ .
  \]
  Then for $1\leq r\leq k$ define
  \[
  \varf_{k,r}\colon \Delta_k \to S(T_{r})\ ,\qquad \sigma\mapsto
  \sum_{m\in \sigma} f_{r}(m)\ .
  \]
  Let $\sigma,\tau\in\Delta_k$. We will show that
  \[
  \frac{r}{k}\ d_\symdif(\sigma,\tau)\leq
  \norm{\varf_{k,r}(\sigma)-\varf_{k,r}(\tau)} \leq
  d_\symdif(\sigma,\tau)\ .
  \]
  Let $i_1<\dots <i_s$ and $j_1<\dots <j_t$ be the elements of
  $\sigma\setminus \tau$ and $\tau\setminus\sigma$, respectively. We
  need to show that
  \[
  \frac{r}{k} (s+t)\leq \norm{f_{r}(i_1)+\dots +f_{r}(i_s)-
    f_{r}(j_1) - \dots - f_{r}(j_t)} \leq (s+t)\ .
  \]
  The upper bound follows from the triangle inequality. Indeed, for
  each $m\in\bn$ and for each $t\in T_r$, summing $f_r(m)$ along the
  branch $b_t$ yields the values $-1,0,1$, and hence $f_r(m)$ is of
  norm~$1$. To see the lower bound, first note that we
  can assume without loss of generality that $1\leq s$ and that either
  $t=0$ or $i_1<j_1$. We will then prove the following statement by
  induction on $\max\{s,t\}$. Given $s+t$ distinct positive integers
  $i_1<\dots <i_s$ and $j_1<\dots <j_t$, where $1\leq s\leq k$ and
  either $t=0$ or $1\leq t\leq k$ and $i_1<j_1$, setting
  \[
  g=f_{r}(i_1)+\dots +f_{r}(i_s)- f_{r}(j_1) - \dots - f_{r}(j_t)\ ,
  \]
  there is a branch functional $\beta_{\ell_1,\dots,\ell_u}$ with
  $1\leq u\leq r$ and $i_1\leq\ell_1$ such that
  \[
  \abs{\beta_{\ell_1,\dots,\ell_u}(g)}=\Bigabs{\sum_{v=1}^u
  e^*_{\ell_1,\dots,\ell_v}(g)} \geq \frac{r}{\max\{r,s,t\}} (s+t)\ .
  \]
  This clearly implies the lower bound of $\frac{r}{k}(s+t)$ on the
  norm of $g$.

  If $s\leq r$ or $r\leq s\leq t$, then for $u=\min\{r,s\}$ we have
  \begin{align*}
    \beta_{i_1,\dots,i_u}(g) & = e^*_{i_1}\Big(f_{r}(i_1)+\sum_{m=2}^s
    f_r(i_m)-\sum_{n=1}^t f_r(j_n)\Big)+ \sum_{v=2}^u
    e^*_{i_1,...,i_v}\big(f_r(i_v)\big)\\
    &=1-(s-1)+t+2(u-1)=-s+t+2u\ .
  \end{align*}  
  When $s\leq r$, then $-s+t+2u=s+t$, and we are done. If $r\leq s\leq
  t$, then
  \[
  -s+t+2u\geq 2r=\frac{r}{\frac12(s+t)} (s+t)\geq
  \frac{r}{\max\{r,s,t\}} (s+t)\ ,
  \]
  as required. We finally deal with the case when $r<s$ and $t<s$. Set
  \[
  h=f_{r}(i_2)+\dots +f_{r}(i_s)- f_{r}(j_1) - \dots - f_{r}(j_t)\ .
  \]
  If $t=0$ or $i_2<j_1$, then we apply the induction hypothesis to
  $h$, and if $j_1<i_2$, then we apply the induction hypothesis to
  $-h$. In either case we obtain a branch functional
  $\beta_{\ell_1,\dots,\ell_u}$ such that $i_1<\ell_1$ and
  $\abs{\beta_{\ell_1,\dots,\ell_u}(h)}\geq \frac{r}{s-1}
  (s+t-1)$. Since $i_1<\ell_1$, we have
  $\beta_{\ell_1,\dots,\ell_u}\big(f_r(i_1)\big)=0$, and it follows
  that
  \[
  \abs{\beta_{\ell_1,\dots,\ell_u}(g)}=\abs{\beta_{\ell_1,\dots,\ell_u}(h)}
  \geq  \frac{r}{s-1}(s+t-1) \geq \frac{r}{s}(s+t)\ .
  \]
  This completes the proof that
  $\dist(\varf_{k,r})\leq\frac{k}{r}$. Recall that Kalton and
  Lancien~\cite{KaltonLancien2008} proved that every separable metric
  space embeds into $\co$ with distortion at most~$2$. It follows that
  $c_{C([0,\omega^r])}(\Delta_k)\leq \min \big\{ \frac{k}{r},2\big\}$.
\end{proof}

\subsection{$\boldsymbol{C([0,\omega^\omega])}$-distortion of the
  Hamming cube}

It follows from Theorem~\ref{thm:Delta_k-into-ST_r} that each
$\Delta_k$ embeds isometrically into $C([0,\omega^\omega])$. We now
prove a stronger result. Recall that a set $A\subset\bn$ is a
\emph{Schreier set }if $\abs{A}\leq\min A$ (or if $A=\emptyset$). The
Schreier family, the set of all Schreier sets, is denoted by
$\cS_1$. We endow $\cS_1$ with the symmetric difference metric, \ie we
consider $\cS_1$ as a subset of $\Delta_\infty$.
\begin{thm}
  \label{thm:schreier-into-C-omega}
  $(\cS_1,d_{\symdif})$ embeds isometrically into $C([0,\omega^\omega])$.
\end{thm}
\begin{proof}
  Define
  \[
  f_\omega\colon \bn \to S(T)\ ,\qquad m \mapsto \sum_{k=1}^m f_k(m)\
  ,
  \]
  where $f_k$, $k\in\bn$, are the functions defined in
  Theorem~\ref{thm:Delta_k-into-ST_r}. Here we identify $x\in S(T_k)$
  with the sequence in $S(T)\cong \big(\bigoplus_{k=1}^\infty
  S(T_k)\big)_{\co}$ that has $x$ in the $k^{\text{th}}$ co-ordinate
  and zero everywhere else. Thus, more precisely,  $f_\omega(m)$ is
  the sequence
  $\big(f_1(m),\dots,f_m(m),0,0,\dots\big)$. We next define
  \[
  \varf_\omega\colon \cS_1 \to S(T)\ ,\qquad \sigma \mapsto \sum_{m\in
    \sigma} f_\omega(m)\ ,
  \]
  and claim that this is an isometric embedding. As before, this
  amounts to showing that if $\sigma,\tau\in\cS_1$ and $i_1<\dots
  <i_s$ and $j_1<\dots <j_t$ are the elements of $\sigma\setminus\tau$
  and $\tau\setminus\sigma$, respectively, then
  \[
  \norm{f_\omega(i_1)+\dots +f_\omega(i_s)-
    f_\omega(j_1) - \dots - f_\omega(j_t)} = s+t\ .
  \]
  Setting $g=f_\omega(i_1)+\dots +f_\omega(i_s)- f_\omega(j_1) - \dots
  - f_\omega(j_t)$, we have $\norm{g}\leq s+t$ by the triangle
  inequality. Indeed, for each $m\in\bn$ we have
  \[
  \norm{f_\omega(m)}=\max _{1\leq k\leq m} \norm{f_k(m)}=1\ .
  \]
  To get the lower bound, we may assume without loss of generality that
  $i_1<j_1$ (or $t=0$) and consider the $k^{\text{th}}$ component of $g$ in
  $S(T_k)$ where $k=i_1$. We will show that
  \[
  \norm{f_k(i_1)+\dots +f_k(i_s)-
    f_k(j_1) - \dots - f_k(j_t)} = s+t\ .
  \]
  Note that $s\leq \abs{\sigma}\leq\min \sigma\leq i_1=k$. It follows that we
  can get the lower bound $s+t$ by applying the branch functional
  $\beta_{i_1,\dots,i_s}$ in $T_k$ as in the proof of
  Theorem~\ref{thm:Delta_k-into-ST_r}.
\end{proof}

We now turn our attention to the infinite Hamming cube. With the help
of Theorem~\ref{thm:schreier-into-C-omega} we are now able to embed
the infinite Hamming cube into $C([0,\omega^\omega])$ with arbitrarily
small distortion. We say that $M$ embeds \emph{almost isometrically
}into $N$, denoted by $M\aiembed N$, if for every $\vare>0$
there exist a bi-Lipschitz embedding $f$ from $M$ into $N$ with
$\textrm{dist}(f)\leq 1+\vare$.
\begin{thm}
  \label{thm:Delta-ai-into-C-omega}
  The infinite Hamming cube $\Delta_\infty$ embeds almost isometrically into
  $C([0,\omega^\omega])$.
\end{thm}
\begin{proof}
  As before, we will in fact embed into $C_0([0,\omega^\omega))$ which
  is identified with $S(T)\cong \big(\bigoplus_{k=1}^\infty
  S(T_k)\big)_{\co}$. Fix $\vare>0$. Choose a sequence
  $0=N_0<N_1<N_2<\dots$ of integers satisfying
  \begin{equation}
    \label{eq:admissibility}
    2m\leq \vare N_m\qquad\text{for all }m\geq 0\ .
  \end{equation}
  We next define maps $f, \varf$ similar to $f_\omega,\Delta_\omega$
  but with a different admissibility condition. It will be clear from
  the definition and the proof of
  Theorem~\ref{thm:schreier-into-C-omega} that this new map $\varf$
  will be an isometric embedding when restricted to the class of sets
  $\sigma$ with $\abs{\sigma}\leq N_{\min \sigma}$. We define
  \[
  f \colon \bn \to S(T)\ ,\qquad m \mapsto \sum_{k=1}^{N_m} f_k(m)\ ,
  \]
  and
  \[
  \varf \colon \Delta_\infty \to S(T)\ ,\qquad \sigma \mapsto
  \sum_{m\in \sigma} f(m)\ ,
  \]
  Fix $\sigma,\tau\in\Delta_\infty$. We will show that
  \begin{equation}
    \label{eq:ai}
    (1-\vare)d_\symdif(\sigma,\tau)\leq
    \norm{\varf(\sigma)-\varf(\tau)}\leq d_\symdif(\sigma,\tau)\ .
  \end{equation}
  By the triangle inequality, we have
  \[
  \norm{\varf(\sigma)-\varf(\tau)}=\Bignorm{\sum_{m\in
      \sigma}f(m)-\sum_{m\in \tau}f(m)}
  \leq \sum_{m\in \sigma\symdif \tau} \norm{f(m)}
  =d_\symdif(\sigma,\tau)\ .
  \]
  To show the lower bound, we first observe that $\sigma$ and $\tau$ can be
  assumed to be disjoint. Indeed, we have
  \[
  \varf(\sigma)-\varf(\tau)=\varf(\sigma\setminus \tau)-\varf(\tau\setminus
  \sigma)\quad\text{and}\quad
  d_\symdif(\sigma,\tau)=d_\symdif(\sigma\setminus \tau,\tau\setminus
  \sigma)\ ,
  \]
  and so we can replace $\sigma$ and $\tau$ with $\sigma\setminus \tau$ and
  $\tau\setminus \sigma$ if necessary.

  We next choose $m,n\in\bn$ such that
  \begin{equation}
    \label{eq:choice-of-m-n}
    N_{m-1}<\abs{\sigma}\leq N_m\qquad\text{and}\qquad
    N_{n-1}<\abs{\tau}\leq N_n\ .
  \end{equation}
  Set $\sigma'=\sigma\setminus \{1,\dots,m-1\}$ and $\tau'=\tau\setminus
  \{1,\dots,n-1\}$. Since $\sigma'$ and $\tau'$ are admissible, we have
  \begin{equation}
    \label{eq:equality-for-admissible}
    \norm{\varf(\sigma')-\varf(\tau')} = d_\symdif(\sigma',\tau')\ .
  \end{equation}
  Next, since $\sigma'$ and $\tau'$ are small perturbations of $\sigma$ and
  $\tau$, respectively, we have
  \begin{multline}
    \label{eq:perturb-phi}
    \bigabs{\norm{\varf(\sigma)-\varf(\tau)}-\norm{\varf(\sigma')-\varf(\tau')}}
    \leq
    \norm{\varf(\sigma)-\varf(\sigma')} + \norm{\varf(\tau)-\varf(\tau')}\\[2ex]
    \leq  d_\symdif(\sigma,\sigma')+d_\symdif(\tau,\tau') \leq (m-1)+(n-1)
  \end{multline}
  and
  \begin{equation}
    \label{eq:perturb-sets}
    \bigabs{d_\symdif(\sigma,\tau)-d_\symdif(\sigma',\tau')}\leq
    d_\symdif(\sigma,\sigma')+d_\symdif(\tau,\tau') \leq (m-1)+(n-1)\
    .
  \end{equation}
  It follows that
  \[
  \begin{array}{rcl@{\quad}l}
    \norm{\varf(\sigma)-\varf(\tau)} & \geq &
    \norm{\varf(\sigma')-\varf(\tau')} -
    (m+n-2) & (\text{by~\eqref{eq:perturb-phi}})\\[2ex]
    &=& d_\symdif(\sigma',\tau') - (m+n-2) &
    (\text{by~\eqref{eq:equality-for-admissible}})\\[2ex]
    &\geq& d_\symdif(\sigma,\tau) - 2(m+n-2) &
    (\text{by~\eqref{eq:perturb-sets}})\\[2ex]
    &=& \abs{\sigma}+\abs{\tau} - 2(m+n-2) &\\[2ex]
    &=& \abs{\sigma} \Big( 1-\frac{2(m-1)}{\abs{\sigma}}\Big)
    +\abs{\tau} \Big( 1-\frac{2(n-1)}{\abs{\tau}}\Big) &\\[2ex]
    &\geq & (1-\vare)(\abs{\sigma}+\abs{\tau}) =
    (1-\vare)\,d_\symdif(\sigma,\tau) &
    (\text{by~\eqref{eq:admissibility} and~\eqref{eq:choice-of-m-n}})
  \end{array}
  \]
  as required.
\end{proof}

\begin{rem}
  An interesting question presents itself in light of the two theorems
  above. Does $\Delta_\infty$ almost isometrically embed into $\cS_1$?
  A positive answer with Theorem~\ref{thm:schreier-into-C-omega} would
  provide another proof of  Theorem~\ref{thm:Delta-ai-into-C-omega}.
\end{rem}

\section{Estimating the $ C(K)$-distortion from
  below}\label{lowerbound}

\subsection{Aharoni's lower bound observed with ``metric lenses''}
Aharoni proved that $c_{\co}(\sepb)\geq 2$, and hence
$c_{\co}(\sepm)\geq 2$. Indeed, he showed that the separable Banach
space $\ell_1$ does not embed into $\co$ with distortion strictly less
than $2$. A careful inspection of his proof shows that the proof and
the statement of the result can be carried out and stated without
using or even mentioning the linear structure of the Banach space
$\ell_1$. This simple but crucial observation allows us to extend
Aharoni's proof to the much more general setting of embeddings into
$C(K)$-spaces.

Denote by $\Deltat_2$ the subset
$\big\{\emptyset,\{n\},\{1,i\},\{2,j\}:\, n\geq 1, i\geq 2 ,j\geq
3\big\}$ of the metric space $\Delta_{2}$. The following theorem is
nothing else but Aharoni's lower bound theorem reformulated in purely
metric terms. For the sake of completeness we include the original
proof using our notation in the hope that it will make the notation
used in the proof of Theorem~\ref{index} more accessible.
\begin{thm}[Aharoni]
  \label{Aharoni}
  The metric space $\Deltat_2$ does not embed into
  $\co$ with distortion strictly less than $2$.
\end{thm}
\begin{proof}
  Assume that $f\colon \Deltat_2\to \co$ and $C<2$ satisfy
  \[
  d_\symdif(\sigma,\tau)\leq\norm{f(\sigma)-f(\tau)}\leq
  Cd_\symdif(\sigma,\tau)\qquad\text{for all }\sigma,\tau\in
  \Deltat_2\ .
  \]
  Without loss of generality one can assume that
  $f(\emptyset)=0$. Let $f_n=e^*_n\circ f$ so that
  $f(\sigma)=\big(f_n(\sigma)\big)_{n=1}^\infty$ for $\sigma\in
  \Deltat_2$. For every $i\neq j$ in $\bn$ define
  \[
  \cX_{i,j}=\{n\in\bn:\,\norm{f_n(\{i\})-f_n(\{j\})}\geq 4-2C\}\ .
  \]
  Note that these are finite sets. Moreover, for every $i,j\geq 3$,
  $i\neq j$, $\cX_{1,2}\cap \cX_{i,j}\neq \emptyset$. Indeed, we have
  \[
  \norm{f(\{1,i\})-f(\{2,j\})}\geq d_\symdif(\{1,i\},\{2,j\})=4\ .
  \]
  Hence there exists $n_{i,j}\in\bn$ such that
  \[
  \norm{f_{n_{i,j}}(\{1,i\})-f_{n_{i,j}}(\{2,j\})}\geq 4\ .
  \]
  It follows that
  \begin{align*}
    \abs{f_{n_{i,j}}(\{i\})-f_{n_{i,j}}(\{j\})} &\geq
    \abs{f_{n_{i,j}}(\{1,i\})-f_{n_{i,j}}(\{2,j\})}\\
    &\quad -
    \abs{f_{n_{i,j}}(\{1,i\})-f_{n_{i,j}}(\{i\})} -
    \abs{f_{n_{i,j}}(\{2,j\})-f_{n_{i,j}}(\{j\})}\\
    &\geq 4-\norm{f(\{1,i\})-f(\{i\})}-\norm{f(\{2,j\})-f(\{j\})}\\
    & \geq 4-Cd_\symdif(\{1,i\},\{i\})-Cd_\symdif(\{2,j\},\{j\})=4-2C\
    .
  \end{align*}
  This proves that $n_{i,j}\in \cX_{i,j}$. Arguing along the same
  lines, one gets that $n_{i,j}\in \cX_{1,2}$ as well. Therefore
  $\cX_{1,2}\cap \cX_{i,j}\neq\emptyset$ whenever $i\neq j$, $i,j\geq
  3$. Denote by $P$ the canonical projection from $\co$ onto the
  closed linear span $Y$ of the vectors $(e_n)_{n\in \cX_{1,2}}$. We
  now obtain a contradiction by observing that the
  sequence $\left(P f(\{n\})\right)_{n=3}^{\infty}$ is a $C$-bounded and
  $(4-2C)$-separated sequence in the finite-dimensional Banach space
  $Y$. Indeed, for every $n\geq 3$,
  \begin{align*}
    \norm{P f(\{n\})}\leq
    \norm{f(\{n\})}=\norm{f(\{n\})-f(\emptyset)}\leq
    Cd_\symdif(\{n\},\emptyset)=C\ ,
  \end{align*}
  and for every $i\neq j$, $i,j\geq 3$, we have
  \begin{align*}
    \norm{P f(\{i\})-P f(\{j\})}& =\sup_{n\in
    \cX_{1,2}}\abs{f_n(\{i\})-f_n(\{j\})}\\
    &\geq \abs{f_{n_{i,j}}(\{i\})-f_{n_{i,j}}(\{j\})}\\
    & \geq 4-2C>0\ .
  \end{align*}
\end{proof}

\subsection{Estimating the $C(K)$-distortion of $\Delta_k$ from below}
A key ingredient in estimating from below the $C(K)$-distortion of the
metric space $\Delta_k$ is the Cantor-Bendixson derivation for
compact spaces. We next recall the definition and a few basic
properties of this derivation.

Let $K$ be a compact topological space. The
\emph{Cantor-Bendixson derivative }$K'$ of $K$ is the set of all
accumulation points of $K$, \ie
\[
K'=K\setminus\{x\in K:\, x\text{ is an isolated point} \}\ .
\]
By transfinite induction one can define derivatives $K^{(\alpha)}$ of
higher order $\alpha$ as follows. We set $K^{(0)}=K$. For an ordinal
$\alpha$ we let $K^{(\alpha+1)}=\big(K^{(\alpha)}\big)'$ and, finally,
for a non-zero limit ordinal $\lambda$ we define
$K^{(\lambda)}=\bigcap_{\alpha<\lambda}K^{(\alpha)}$.

We gather in the next proposition some basic properties of the
Cantor-Bendixson derivation.
\begin{prop}
  \label{prop:cb-properties}
  Let $K$ be a compact metric space. Then
  \begin{enumerate}
  \item $K$ is finite $\Longleftrightarrow$ $K'=\emptyset$
    $\Longleftrightarrow$ $K$ is discrete;
  \item $K$ is countable $\Longleftrightarrow$ $\E\ \alpha<\omega_1$
    such that $K^{(\alpha)}=\emptyset$;
  \item $K$ is uncountable $\Longleftrightarrow$ $\exists\
    \alpha<\omega_1$ such that
    $K^{(\alpha+1)}=K^{(\alpha)}\neq\emptyset$.
  \end{enumerate}
\end{prop}
For a general compact topological space $K$ the smallest ordinal
$\alpha$ such that $K^{(\alpha)}=K^{(\alpha+1)}$ is called the
\emph{Cantor-Bendixson index }(or \emph{rank}) of $K$, and we denote
it by $\cbi(K)$. For example, consider the compact space
$K=[0,\omega^\alpha\cdot n]$, where $1\leq \alpha<\omega_1$ and $1\leq
n<\omega$. Then $\cbi(K)=\alpha+1$ and $\abs{K^{(\alpha)}}=n$. More
generally, if $K$ is a countably infinite compact metric space, then
for some $1\leq\alpha<\omega_1$ and $1\leq n<\omega$ we have
$\cbi(K)=\alpha+1$, $\abs{K^{(\alpha)}}=n$ and $K$ is homeomorphic to
$[0,\omega^\alpha\cdot n]$. Thus, the Cantor-Bendixson derivation
gives rise to a topological classification of countable compact metric
spaces, and hence an isometric classification of $C(K)$-spaces with
separable dual.

Inspired by the reformulation of Aharoni's proof in terms of a metric
subset of $\Delta_{2}$ we establish a link between the
$C(K)$-distortion of the sequence $(\Delta_k)_{k\geq 1}$ and the
Cantor-Bendixson index of the compact space $K$. In
Section~\ref{upperbound} we showed that
$c_{C([0,\omega^{k-1}])}(\Delta_k)\leq \frac{k}{k-1}$ for $k\geq
2$. In the remainder of this section we will show that the upper bound
is tight.
\begin{thm}
  \label{index}
  Let $K$ be a compact topological space and $k$ be an integer with
  $k\geq 2$. If $\Delta_k$ admits a bi-Lipschitz embedding into $C(K)$
  with distortion strictly less than $\frac{k}{k-1}$, then
  $\cbi(K)\geq k+1$. It follows that
  $c_{C([0,\omega^{k-1}])}(\Delta_k)=\frac{k}{k-1}$.
\end{thm}
\begin{proof}
  Assume that there is a function  $f\colon \Delta_k\to  C(K)$ and
  a constant $D<\frac{k}{k-1}$ such that
  \[
  d_\symdif(\sigma,\tau)\leq \norm{f(\sigma)-f(\tau)} \leq
  Dd_\symdif(\sigma,\tau)\qquad\text{for all } \sigma, \tau\in
  \Delta_k .
  \]
  Set $\eta=2k-2(k-1)D$ and observe that $\eta>0$. For distinct
  $i,j\in\bn$ define
  \[
  \cX_{i,j}=\big\{ \beta\in K:\,
  \abs{f(\{i\})(\beta)-f(\{j\})(\beta)}\geq \eta\big\}\ .
  \]
  Consider the following statment. For each $0\leq s\leq k$ and for
  any $2(k-s)$ distinct integers $i_1,i_2,\dots, i_{k-s}, j_1,
  j_2,\dots,j_{k-s}$, we have
  \[
  K^{(s)}\cap \cX_{i_1,j_1}\cap\cX_{i_2,j_2}\cap\dots\cap
  \cX_{i_{k-s},j_{k-s}}\neq\emptyset\ .
  \]
  We will now verify this statement by induction on $s$. The theorem
  will then follow by putting $s=k$.

  We begin with $s=0$. Let $i_1,\dots,i_k$ and $j_1,\dots,j_k$ be $2k$
  distinct elements of $\bn$. Set $\sigma=\{i_1,\dots,i_k\}$ and
  $\tau=\{j_1,\dots,j_k\}$. Since $\norm{f(\sigma)-f(\tau)}\geq
  d_{\symdif}(\sigma,\tau)= 2k$, there exists $\beta\in K$ such that
  $\abs{f(\sigma)(\beta)-f(\tau)(\beta)}\geq 2k$. It follows that
  \begin{align*}
    \bigabs{f(\{i_r\})(\beta)-f(\{j_r\})(\beta)} & \geq
    \bigabs{f(\sigma)(\beta)-f(\tau)(\beta)} -
    \bigabs{f(\sigma)(\beta)-f(\{i_r\})(\beta)} \\
    & \quad - \bigabs{f(\{j_r\})(\beta)-f(\tau)(\beta)}\\
    & \geq 2k - \norm{f(\sigma)-f(\{i_r\})} -
    \norm{f(\{j_r\})-f(\tau)}\\
    & \geq 2k - Dd_{\symdif}(\sigma,\{i_r\}) -
    Dd_{\symdif}(\{j_r\},\tau)\\
    &\geq 2k-2D(k-1)=\eta>0
  \end{align*}
  for each $1\leq r\leq k$. Thus, $\beta\in \cX_{i_1,j_1}\cap\dots\cap
  \cX_{i_k,j_k}$.

  Now assume that the statement holds for some $0\leq s<k$. Let
  $i_1,\dots, i_{k-s-1}$ and $j_1,\dots,j_{k-s-1}$ be $2(k-s-1)$ distinct
  elements of $\bn$. Let
  \[
  L= K^{(s)}\cap \cX_{i_1,j_1}\cap\dots\cap
  \cX_{i_{k-s-1},j_{k-s-1}}
  \]
  and $n_0=\max \{i_1,\dots,i_{k-s-1},j_1,\dots,j_{k-s-1}\}$. Note that $L$ is a
  closed subset of $K$. Let us denote by $R$ the restriction operator
  $C(K)\to C(L)$. Note that for distinct $i,j>n_0$ we have $L\cap
  \cX_{i,j}\neq\emptyset$ by the induction hypothesis. It follows that
  the functions $Rf(\{i\})$, $i>n_0$, are uniformly bounded
  and $\eta$-separated. Indeed, we have
  \begin{multline*}
    \norm{Rf(\{i\})}\leq \norm{f(\{i\})} \leq
    \norm{f(\{i\})-f(\emptyset)}+\norm{f(\emptyset)} \leq D
    d_{\symdif}(\{i\},\emptyset) + \norm{f(\emptyset)}\\ \leq D
    +\norm{f(\emptyset)}\ ,
  \end{multline*}
  and for distinct $i,j>n_0$ we can pick $\beta\in L\cap\cX_{i,j}$ and
  obtain
  \[
  \bignorm{Rf(\{i\})-Rf(\{j\})} \geq
  \bigabs{f(\{i\})(\beta)-f(\{j\})(\beta)} \geq  \eta\ .
  \]
  We deduce that $C(L)$ must be infinite-dimensional, and hence $L$
  must be infinite. Since every infinite compact space has an
  accumulation point, we have
  \[
  K^{(s+1)}\cap \cX_{i_1,j_1}\cap\cX_{i_2,j_2}\cap\dots\cap
  \cX_{i_{k-s-1},j_{k-s-1}}\neq\emptyset\ ,
  \]
  as required.
\end{proof}

\section{Applications and open problems}
\label{applications}

Let $(M,d_M)$ denote an arbitrary separable metric space. Let
$D_\alpha$ be an upper bound on $c_{C([0,\omega^\alpha])}(M)$, and
consider the following self-explanatory diagram.
\begin{tiny}
  \[
  \begin{array}{ccccccccc}
    \co&\isometric&C([0,\omega])&\isometric\cdots\isometric&C([0,\omega^k])&\isometric\cdots\isometric&C([0,\omega^\omega])&\isometric &C([0,1])\\
    \uptwolip& &\updonelip & & \updklip & & \updwlip&
    &\uponelip\\
    (M,d_M)      & & (M,d_M)      & \cdots  &   (M,d_M)    & \cdots& (M,d_M)     & &(M,d_M)\\
  \end{array}
  \]
\end{tiny}
Whereas the best distortion achievable in the two extreme cases is
completely understood, essentially no estimates for the values of the
parameters $c_{C([0,\omega^\alpha])}(\cC)$ have been hitherto known for
$\cC$ being any class among $\sepm, \sepb, \cotype, \type, \sepref$
(besides the upper bound $2$ which follows from
Kalton-Lancien embedding result~\cite{KaltonLancien2008}). It is worth
noting that since any $C(K)$-space (for $K$ countable) is
$\co$-saturated, it cannot be a \emph{linearly }isometrically
universal space for the class of separable Banach spaces. Moreover, it
cannot be an isometrically universal space either since Godefroy and
Kalton~\cite{GodefroyKalton2003} proved that if a separable Banach
space $X$ embeds isometrically into a Banach space $Y$, then $Y$
contains an linear isometric copy of $X$. Our study of stratifications
of the Hamming cube (Theorem~\ref{index}) yields nontrivial lower
bounds for the first time.
\begin{cor}
  \label{cklower}
  Let $\cC\in\{\sepm, \sepb, \cotype, \type,\sepref\}$, and let
  $k\in\bn$. Then $\frac{k+1}{k}\leq c_{C([0,\omega^k])}(\cC)\leq 2$.
\end{cor}
\begin{proof}
  We first remark that the upper bound for all $k$ is the result of
  Kalton and Lancien~\cite{KaltonLancien2008}, and the lower bound for
  $k=1$ is due to Aharoni~\cite{Aharoni1974}. We now consider the
  lower bound for $k\geq 2$.

  Set $K=[0,\omega^k]$, and note that $\cbi(K)=k+1$. It follows from
  Theorem~\ref{index} that $c_{C(K)}(\Delta_{k+1})\geq \frac{k+1}{k}$.
  Given $\vare>0$, choose $p$ with $1<p<\infty$ such that the function
  $f\colon \Delta_{k+1}\to \ell_p$ defined by $f(\sigma)=\sum_{i\in
    \sigma} e_i$ is a $(1+\vare)$-isometric embedding. It follows that
  $c_{C(K)}(\Delta_k)\leq (1+\vare)c_{C(K)}(\ell_p)$. Since $\ell_p$
  belongs to the class $\cC$, we have $c_{C(K)}(\Delta_k)\leq
  (1+\vare)c_{C(K)}(\cC)$, and the result is proved.
\end{proof}
The following corollary is an easy consequence of Theorem~\ref{index}
and the fact that $(\Delta_k)_{k\geq 1}$ is a stratification of
$H_\infty$.
\begin{cor}
  \label{ai}
  Let $K$ be a countable compact metric space.
  If $H_\infty\aiembed C(K)$, then $\cbi(K)\geq \omega+1$. In
  particular, if $C(K)$ is an almost isometrically universal space for
  the class $\cC\in\{\sepm, \sepb, \cotype, \type, \sepref\}$, then
  $\cbi(K)\geq \omega+1$.
\end{cor}
\begin{proof}
  It follows from Theorem~\ref{index} that $\cbi(K)\geq k+1$ for every
  $k<\omega$, and hence
  $K^{(\omega)}=\bigcap_{k<\omega} K^{(k)}\neq\emptyset$. The result
  follows by Proposition~\ref{prop:cb-properties}.
\end{proof}
\begin{rem}
  Prochazka and S\'anchez-Gonz\'alez~\cite{PSG} using the technique of
  Section~\ref{lowerbound} exhibited a countable nonproper metric space
  which does not admit an embedding with distortion less than $2$ into
  any $C(K)$-space with $K$ countable. Therefore for such compact spaces
  $K$ we have $c_{C(K)}(\sepm)=c_{C(K)}(\sepb)=2$, and hence $C(K)$
  cannot be an almost isometrically universal space for the  classes
  $\sepm$ or $\sepb$.
\end{rem}

The following theorem, of independent interest, can also be used to
prove the second part of Corollary~\ref{ai} in combination with either
Aharoni's original lower bound involving $\ell_1$ or
Corollary~\ref{cklower}.
\begin{thm}
  If $\ell_1\aiembed  C(K)$ then $\ell_1\aiembed  C(K^{(\alpha)})$ for
  all ordinals $\alpha<\omega$.
\end{thm}
\begin{proof}
  It is sufficient to show that if $\ell_1\aiembed  C(K)$, then
  $\ell_1\aiembed  C(K')$. Fix $\vare>0$ and let $f\colon
  C(K)\to\ell_1$ be a function satisfying
  \[
  \frac{\norm{x-y}_1}{1+\vare} \leq \norm{f(x)-f(y)}_\infty
  \leq\norm{x-y}_1\ .
  \]
  Define $g\colon \ell_1\to C(K')$ by letting $g(x)$ be the
  restriction of $f(x)$ to $K'$ ($x\in\ell_1$).  We are going to show
  that $\dist(g)\leq \frac{1+\vare}{1-2\vare}$, which then completes
  the proof.

  Fix distinct vectors $x,y\in\ell_1$ of finite support. Let
  $\delta=\norm{x-y}_1$ and $n_0=\max \supp(x)\cup\supp(y)$. For
  distinct integers $i,j>n_0$ we have
  \[
  \norm{f(x+\delta e_i)-f(y+ \delta
    e_j)}_{\infty}\geq\frac{3\delta}{1+\vare}\ .
  \]
  Hence there exists $\beta\in K$ such that
  \begin{equation}
    \label{equa1}
    \abs{f(x+\delta e_i)(\beta)-f(y+ \delta e_j)(\beta)}
    \geq\frac{3\delta}{1+\vare}\ .
  \end{equation}
  We next observe that if~\eqref{equa1} holds, then we also have
  \begin{equation}
    \label{eq:separated-seq}
    \abs{f(x+\delta e_i)(\beta)-f(x+\delta e_j)(\beta)}\geq
    \frac{(2-\vare)\delta}{1+\vare}\ ,
  \end{equation}
  and
  \begin{equation}
    \label{eq:separated-vec}
    \abs{f(x)(\beta)-f(y)(\beta)} \geq
  \frac{(1-2\vare)\delta}{1+\vare} =
  \frac{(1-2\vare)\norm{x-y}_1}{1+\vare}\ .
  \end{equation}
  Now let
  \[
  L=\big\{ \beta\in K:\,\E\text{ distinct } i, j>n_o \text{ satisfying
    equation }\eqref{equa1}\big\}\ .
  \]
  For $z\in\ell_1$ let $f_L(z)$ denote the restriction of $f(z)$ to
  $L$. By~\eqref{eq:separated-seq}, the sequence $\big(f_L(x+\delta
  e_i)\big)_{i>n_0}$ in $C(L)$ is bounded and
  $\frac{(2-\vare)\delta}{1+\vare}$-separated. It follows that $L$ is
  infinite, and so $L\cap
  K'\neq\emptyset$. By~\eqref{eq:separated-vec}, for any $\beta\in
  L\cap K'$ we have $\abs{f(x)(\beta)-f(y)(\beta)}\geq
  \frac{(1-2\vare)\norm{x-y}_1}{1+\vare}$. Thus
  \[
  \norm{g(x)-g(y)}\geq \frac{(1-2\vare)\norm{x-y}_1}{1+\vare}\ .
  \]
  This shows that $g\colon\ell_1\to C(K')$ is a bi-Lipschitz embedding
  with constant $\frac{1+\vare}{1-2\vare}$, as claimed.
\end{proof}
We conclude by stating some open problems. In light of the above
result, it is natural to ask the following.
\begin{qn}
  Does $\ell_1$ almost isometrically embed into
  $C([0,\omega^\omega])$?
\end{qn}
Recall that one cannot hope for an isometric embedding because of the
aformentioned result of Godefroy and Kalton~\cite{GodefroyKalton2003}. 

Recall also that using the techniques of Theorem~\ref{index}, Prochazka
and S\'anchez-Gonz\'alez~\cite{PSG} constructed a separable metric
space $M$ for which $c_{C(K)}(M)=2$ for any (infinite) countable
compact space $K$. However, it is not clear whether their example
embeds into $\ell_1$ isometrically (or with distortion less
than~$2$). Indeed, it is not known if their example isometrically
embeds into any Banach space which is not already universal for
$\sepb$. So the following open problems seem to be of interest.
\begin{qn}
  Is there some non-trivial class $\cC$ of Banach spaces and a
  countable compact space $K$ such that $C(K)$ is almost isometrically
  universal for the class $\cC$?
\end{qn}
The above question is deliberately vague. Examples we have in mind
for non-trivial classes include $\type, \cotype$ and $\sepref$. We
conclude with a more specific quantitative question.
\begin{qn}
  Let $\alpha\in[2,\omega_1)$. What is the exact value of
  $c_{C([0,\omega^\alpha])}(\cC)$ for $\cC\in\{\type, \cotype,
  \sepref\}$?
\end{qn}

\end{document}